\documentclass[12pt,twoside]{amsart}
\usepackage{amsmath}
\usepackage{amsthm}
\usepackage{amsfonts}
\usepackage{amssymb}
\usepackage{latexsym}
\usepackage[all]{xy}
\usepackage{extarrows}

\date{}
\allowdisplaybreaks[4] \footskip=15pt
\renewcommand{\uppercasenonmath}[1]{}

\numberwithin{equation}{section} \theoremstyle{plain}
\newtheorem*{thm*}{Main Theorem}
\newtheorem{thm}{Theorem}[section]
\newtheorem{cor}[thm]{Corollary}
\newtheorem*{cor*}{Corollary}
\newtheorem{lem}[thm]{Lemma}
\newtheorem*{lem*}{Lemma}

\newtheorem*{prop*}{Proposition}
\newtheorem{rem}[thm]{Remark}
\newtheorem*{rem*}{Remark}

\newtheorem*{exa*}{Example}
\newtheorem{df}[thm]{Definition}
\newtheorem*{df*}{Definition}

\newtheorem*{conj*}{Conjecture}
\newtheorem*{ack*}{ACKNOWLEDGEMENTS}

\newcommand{\pf}{\noindent\begin {proof}}
\newcommand{\epf}{\end{proof}}

\newcommand{\Ext}{\mbox{\rm Ext}}

\newcommand{\Hom}{\mbox{\rm Hom}}
\newcommand{\Tor}{\mbox{\rm Tor}}
\newcommand{\im}{\mbox{\rm im}}
\newcommand{\coker}{\mbox{\rm coker}}
\begin{document}
\begin{center}
{\bf  Gorenstein homological properties of $n$-trivial extensions of rings}

\vspace{0.5cm} Lixin Mao\\
School of Mathematics and Physics, Nanjing Institute of Technology,\\ Nanjing 211167, China\\
E-mail: maolx2@hotmail.com \\
\end{center}

\bigskip
\centerline { \bf  Abstract}
 \bigskip
\leftskip10truemm \rightskip10truemm
 \noindent
We explicitly investigate  Gorenstein projective, injective and flat modules over the $n$-trivial extension $R\ltimes_{n}M$ of a ring $R$ by an $R$-bimodule $M$. Assume that  $fd(M^{\otimes_{R}i}_{R})<\infty$ and $pd(_{R}M^{\otimes_{R}i})<\infty$ for any $1\leq i\leq n$, $fd(\textbf{Z}(R)_{R\ltimes_{n}M})<\infty$ and $pd(_{R\ltimes_{n}M}\textbf{Z}(R))<\infty$. It is proven that a left $R\ltimes_{n}M$-module $(X,f)$ is Gorenstein projective if and only if the sequence  $M^{\otimes_{R}n+1}\otimes_{R} X\stackrel{(M\otimes f) \cdots(M^{\otimes_{R}n}\otimes f)}\longrightarrow  M\otimes_{R} X\stackrel{f}\longrightarrow X$   is exact and coker($f$) is a Gorenstein projective left $R$-module. As a consequence, we characterize Gorenstein projective, injective and flat modules over tensor rings.\\
\vbox to 0.3cm{}\\
{\it Key Words:} $n$-trivial extension; Gorenstein projective module; Gorenstein injective module; Gorenstein flat module.\\
{\it 2020 Mathematics Subject Classification:} 16D40; 16D50; 16D90.

\leftskip0truemm \rightskip0truemm
\bigskip
\section {\bf Introduction}

The origin of Gorenstein homological algebra may date back to 1960s when Auslander and Bridger introduced the concept of G-dimension for
finitely generated modules over a two-sided Noetherian ring \cite{AB}. Later, Enochs, Jenda and Torrecillas extended the
ideas of Auslander and Bridger and introduced the concepts of Gorenstein projective, injective and flat modules over arbitrary
rings \cite{EJ1}-\cite{EJT}, which have significant applications in representation theory of algebras, algebraic geometry and other
fields.

Let $R$ be an associative ring and $M$ an $R$-bimodule. Then the Cartesian product $R \times M$, with the natural addition and the multiplication given by $(r_{1}, m_{1})(r_{2}, m_{2}) = (r_{1}r_{2}, r_{1}m_{2}+m_{1}r_{2})$, becomes a ring. This ring is called the \emph{trivial  extension} of the ring $R$ by the $R$-bimodule $M$, and denoted
by $R\ltimes M$. One important special case of the trivial extension is the formal triangular matrix ring. The notion of trivial extension of a ring by a bimodule is an important extension of rings and has played a crucial role in ring theory and homological algebra \cite{DMZ, FGR, PR, R}. In particular, Gorenstein projective, injective and flat
modules over  formal triangular matrix rings and trivial ring extensions  have been investigated in many papers  \cite{EIT, HJ, M, Z}.

In \cite{ABFS}, Anderson, Bennis, Fahid and  Shaiea introduced the concept of $n$-trivial extension of a ring $R$ by a family $(M_{i})_{i=1}^{n}$ of $n$ $R$-bimodules, denoted by $R\ltimes_{n} M_{1}\ltimes \cdots\ltimes M_{n}$ and  obtained many ring-theoretic properties of this kind of rings. Later, Benkhadra, Bennis and Garc\'{\i}a Rozas studied some basic module-theoretic properties of $R\ltimes_{n} M_{1}\ltimes \cdots\ltimes M_{n}$ in \cite{BBG}. Recently, the author focused on a salient and tractable  subclass, i.e., the $n$-trivial extension $R\ltimes_{n} M$ of a ring $R$ by an $R$-bimodule $M$ in \cite{M1}, which is not only a generalization of  the classical trivial extension  $R\ltimes M$, but also a quotient ring of the tensor ring $T_{R}(M)$ by the ideal generated by $M^{\otimes_{R}n+1}$ in the sense of \cite{Co}. Many interesting homological properties of the ring $R\ltimes_{n} M$ were obtained. In the present paper, we will further investigate the transfers of Gorenstein properties between an arbitrary ring $R$ and its $n$-trivial ring extension $R\ltimes_{n} M$. As a consequence, we give characterizations of Gorenstein projective, injective and flat modules over tensor rings, which extend some related results in \cite{DLST,TW}.

The contents of this paper are arranged as follows.

In Section 2, we recall some basic concepts and results on $n$-trivial extensions.

In Section 3,  we describe Gorenstein projective and injective $R\ltimes_{n} M$-modules. For example, assume that $fd(M^{\otimes_{R}i}_{R})<\infty$ and $pd(_{R}M^{\otimes_{R}i})<\infty$ for any $1\leq i\leq n$, $fd(\textbf{Z}(R)_{R\ltimes_{n}M})<\infty$ and $pd(_{R\ltimes_{n}M}\textbf{Z}(R))<\infty$. It is proven that $(X,f)$ is a Gorenstein projective left $R\ltimes_{n}M$-module if and only if the sequence  $M^{\otimes_{R}n+1}\otimes_{R} X\stackrel{(M\otimes f) \cdots(M^{\otimes_{R}n}\otimes f)}\longrightarrow  M\otimes_{R} X\stackrel{f}\longrightarrow X$   is exact and $\coker(f)$ is a Gorenstein projective left $R$-module (see Theorems \ref{thm: 3.1} and \ref{thm: 3.2}).

Section 4 is devoted to Gorenstein flat and $PGF$ $R\ltimes_{n} M$-modules. For example, assume that $fd(_{R}M^{\otimes_{R}i})<\infty$, $fd(M^{\otimes_{R}i}_{R})<\infty$ for any $1\leq i\leq n$, $fd(_{R\ltimes_{n}M}\textbf{Z}(R))<\infty$, $fd(\textbf{Z}(R)_{R\ltimes_{n}M})$ or $id(\textbf{Z}(R)_{R\ltimes_{n}M})<\infty$. We prove that $(X,f)$ is a $PGF$ left $R\ltimes_{n}M$-module if and only if the sequence  $M^{\otimes_{R}n+1}\otimes_{R} X\stackrel{(M\otimes f) \cdots(M^{\otimes_{R}n}\otimes f)}\longrightarrow  M\otimes_{R} X\stackrel{f}\longrightarrow X$   is exact and $\coker(f)$ is a $PGF$ left $R$-module (see Theorems \ref{thm: 4.4} and \ref{thm: 4.5}).
\bigskip
\section {\bf Preliminaries}
\bigskip
Throughout this paper, all rings are associative rings with identity and all modules are unitary. For a ring $R$, we write $R$-Mod (resp. Mod-$R)$  for the category of left (resp. right) $R$-modules. $_RX$ (resp. $X_{R}$) denotes a left (resp. right) $R$-module, $pd(_{R}X)$, $id(_{R}X)$ and $fd(_{R}X)$ denote the projective, injective and flat dimensions of $_{R}X$ respectively. The character module $\Hom_{\mathbb{Z}}(X,\mathbb{Q}/\mathbb{Z})$ of $X$  is denoted by $X^{+}$.

We first recall the concept of an $n$-trivial extension $R\ltimes_{n} M$ of a ring $R$ by an $R$-bimodule $M$ \cite{M1}.
\begin{df}\label{df: 2.1} {\rm Let $M$ be an $R$-bimodule and $n\geq 1$. Define a multiplication in the direct sum $\coprod_{i=0}^{n}M^{\otimes_{R}i}$ of Abelian groups,
where $M^{\otimes_{R}0} = R$, $M^{\otimes_{R}1}=M$ and $M^{\otimes_{R}i+1} =  M\otimes_{R}M^{\otimes_{R}i}$, by
$(m_{0}, ..., m_{n})(m_{0}^{'}, ..., m_{n}^{'}) = (\sum_{i+j=k} m_{i}\otimes m_{j}^{'})_{0\leq k\leq n}$. Then $\coprod_{i=0}^{n}M^{\otimes_{R}i}$ becomes a ring. This ring is said to be an \emph{$n$-trivial extension} of the ring $R$ by $M$, denoted by $R\ltimes_{n} M$.}
\end{df}
Notice that $R\ltimes_{1} M$ is nothing but the classical trivial extension $R\ltimes M$.

In order to investigate modules over $R\ltimes_{n} M$, the author introduced the following category $\Omega$ in \cite{M1}: whose objects are couples $(X, f)$ with $X$ a left $R$-module and $f\in\Hom_{R}(M\otimes_{R}X, X)$ such that the composition of $$M^{\otimes_{R}n+1}\otimes_{R} X\stackrel{M^{\otimes_{R}n}\otimes f}\longrightarrow M^{\otimes_{R}n}\otimes_{R} X\longrightarrow\cdots\longrightarrow  M\otimes_{R}  M\otimes_{R} X\stackrel{ M\otimes f}\longrightarrow  M\otimes_{R} X\stackrel{f}\longrightarrow X$$ is zero and a morphism $(X,f) \rightarrow (Y,g)$ is  $\gamma\in\Hom_{R}(X, Y)$ such that the following diagram commutes. $$\xymatrix{M\otimes_{R}X\ar[d]_{f}\ar[r]^{M\otimes\gamma}&M\otimes_{R}Y\ar[d]_{g}\\X\ar[r]^{\gamma}&Y}$$
A sequence $(X_{1},f_{1})\stackrel{\gamma_{1}}\longrightarrow (X_{2},f_{2})\stackrel{\gamma_{2}}\longrightarrow(X_{3},f_{3})$ in $\Omega$ is exact if and only if the underlying sequence $X_{1}\stackrel{\gamma_{1}}\longrightarrow X_{2}\stackrel{\gamma_{2}}\longrightarrow X_{3}$ in $R$-Mod is exact.

In view of the adjointness relation, the category $\Omega$ is equivalent  to the category $\Upsilon$: whose objects are couples $[Y,g]$ with $Y$ a left $R$-module and $g\in\Hom_{R}(Y, \Hom_{R}(M,Y))$ such that the composition of {\Small $$Y\stackrel{g}\longrightarrow \Hom_{R}(M,Y)\stackrel{{\rm Hom}_{R}(M,g)}\longrightarrow \Hom_{R}(M,\Hom_{R}(M,Y))\rightarrow\cdots\rightarrow\Hom_{R}^{n}(M,Y)\stackrel{{\rm Hom}_{R}^{n}(M,g)}\longrightarrow  \Hom_{R}^{n+1}(M,Y)$$} is zero, where $\Hom_{R}^{0}(M,Y)=Y, \Hom_{R}^{1}(M,Y)=\Hom_{R}(M,Y), \cdots,
\Hom_{R}^{n}(M,Y)=\Hom_{R}(M,\Hom_{R}^{n-1}(M,Y))$, and a morphism $[X,f] \rightarrow [Y,g]$ is  $\varphi\in\Hom_{R}(X, Y)$ such that the following diagram commutes. $$\xymatrix{X\ar[d]_{f}\ar[rr]^{\varphi}&&Y\ar[d]_{g}\\\Hom_{R}(M,X)\ar[rr]^{{\rm Hom}_{R}(M,\varphi)}&&\Hom_{R}(M,Y)}$$
A sequence $[Y_{1},g_{1}]\stackrel{\varphi_{1}}\longrightarrow [Y_{2},g_{2}]\stackrel{\varphi_{2}}\longrightarrow [Y_{3},g_{3}]$ in $\Upsilon$ is exact if and only if the underlying sequence $Y_{1}\stackrel{\varphi_{1}}\longrightarrow Y_{2}\stackrel{\varphi_{2}}\longrightarrow Y_{3}$ in $R$-Mod is exact.

There are some important functors as follows.

The functor $\textbf{T}: R$-Mod $\longrightarrow \Omega$ is given, for every object $X\in R$-Mod, by $\textbf{T}(X) =(\coprod_{i=0}^{n}(M^{\otimes_{R}i}\otimes_{R}X), \mu)$, with $\mu(x_{0}, x_{1}, x_{2},\cdots, x_{n})=(0, x_{0}, x_{1},\cdots, x_{n-1})$  and for morphisms by $\textbf{T}(\alpha)=\coprod_{i=0}^{n}(M^{\otimes_{R}i}\otimes\alpha)$.

The functor $\textbf{U}:\Omega\longrightarrow R$-Mod is given, for every object $(X,f)\in \Omega$, by $\textbf{U}(X,f) =X$ and for morphisms by $\textbf{U}(\alpha)=\alpha$. Since the category $\Upsilon$ is  equivalent  to $\Omega$, $\textbf{U}$ can also be interpreted as the functor $\textbf{U}:\Upsilon\longrightarrow R$-Mod: for every object $[Y,g]\in \Upsilon$, $\textbf{U}[Y,g] =Y$ and for morphisms, $\textbf{U}(\beta)=\beta$.

The functor $\textbf{Z}: R$-Mod $\longrightarrow \Omega$ is given, for every object $X\in R$-Mod, by $\textbf{Z}(X)=(X,0)$ and for morphisms by $\textbf{Z}(\alpha)=\alpha$. Since the category $\Upsilon$ is  equivalent  to $\Omega$, $\textbf{Z}$ can also be interpreted as the functor $\textbf{Z}: R$-Mod $\longrightarrow\Upsilon$: for every object  $Y\in R$-Mod, $\textbf{Z}(Y)=[Y,0]$ and for morphisms, $\textbf{Z}(\beta)=\beta$.

The functor $\textbf{C}:\Omega\longrightarrow R$-Mod is given, for every object $(X,f)\in \Omega$, by $\textbf{C}(X,f) =\coker(f)$ and for morphisms by $\textbf{C}(\alpha)=$ the induced morphism.

The functor $\textbf{H}: R$-Mod $\longrightarrow \Upsilon$ is given, for every object  $Y\in R$-Mod, by $\textbf{H}(Y)=[\coprod_{i=0}^{n}\Hom_{R}^{i}(M,Y),\nu]$,  with $\nu(x_{0}, x_{1}, x_{2},\cdots, x_{n})=(x_{1}, x_{2},\cdots, x_{n},0)$ and for morphisms by $\textbf{H}(\beta)=\coprod_{i=0}^{n}\Hom^{i}_{R}(M,\beta)$.

The functor $\textbf{K}: \Upsilon\longrightarrow R$-Mod is given, for every object $[Y,g]\in \Upsilon$, by $\textbf{K}[Y,g] =\ker(g)$ and for morphisms by $\textbf{K}(\beta)=$ the induced morphism.

There are analogous functors for right modules.

By \cite[Proposition 2.3]{M1}, $(\textbf{T}, \textbf{U})$, $(\textbf{U}, \textbf{H})$, $(\textbf{C}, \textbf{Z})$ and $(\textbf{Z}, \textbf{K})$ are adjoint pairs  such that $\textbf{C} \textbf{T}= id_{R{\rm-Mod}}$, $\textbf{U} \textbf{Z}= id_{R{\rm-Mod}}$ and $\textbf{K} \textbf{H}= id_{R{\rm-Mod}}$. Consequently, $\textbf{C}$ and $\textbf{T}$ are right exact, $\textbf{K}$ and $\textbf{H}$ are left exact.

By \cite[Proposition 2.4]{M1}, the category $R\ltimes_{n}M$-Mod is equivalent to the categories $\Omega$ and $\Upsilon$. In the rest of the paper, we will identify  $R\ltimes_{n}M$-Mod  with $\Omega$ and $\Upsilon$.
\bigskip
\section {\bf Gorenstein projective and injective modules over $n$-trivial extensions}
\bigskip
Recall that a  left $R$-module $X$ is \emph{Gorenstein projective} \cite{EJ1} if there is an exact sequence of projective left $R$-modules $$\Xi: \cdots\longrightarrow
P^{-1}\longrightarrow P^{0}\stackrel{g^{0}}\longrightarrow P^{1}\longrightarrow P^{2}\longrightarrow \cdots$$  such that $X\cong\ker(g^{0})$ and $\Hom_{R}(\Xi, Q)$ is also exact for any projective left $R$-module $Q$.

We first  establish  necessary conditions for Gorenstein projective modules over $n$-trivial extensions.
\begin{thm} \label{thm: 3.1}Assume that   $fd(\textbf{Z}(R)_{R\ltimes_{n}M})<\infty$ and $pd(_{R\ltimes_{n}M}\textbf{Z}(R))<\infty$. If $(X,f)$ is a Gorenstein projective left $R\ltimes_{n}M$-module, then the sequence  $M^{\otimes_{R}n+1}\otimes_{R} X\stackrel{(M\otimes f) \cdots(M^{\otimes_{R}n}\otimes f)}\longrightarrow  M\otimes_{R} X\stackrel{f}\longrightarrow X$   is exact and $\coker(f)$ is a Gorenstein projective left $R$-module.
\end{thm}
\begin{proof}There is an exact sequence of projective left  $R\ltimes_{n}M$-modules $$\Delta: \cdots\longrightarrow \textbf{T}(P^{-1})\longrightarrow
\textbf{T}(P^{0})\stackrel{\varphi^{0}}\longrightarrow \textbf{T}(P^{1})\longrightarrow \textbf{T}(P^{2})\longrightarrow \cdots$$
such that $(X,f)\cong\ker(\varphi^{0})$, each $P^{i}$ is a projective left $R$-module by \cite[Theorem 2.8]{M1} and $\Hom_{R\ltimes_{n}M}(\Delta, Q)$ is  exact for any projective left $R\ltimes_{n}M$-module $Q$.

Since $fd(\textbf{Z}(R)_{R\ltimes_{n}M})<\infty$, $\textbf{Z}(R)\otimes_{R\ltimes_{n}M}\Delta$ is exact by \cite[Lemma 2.3]{EIT}. Since $\textbf{Z}(R)\otimes_{R\ltimes_{n}M} \textbf{T}(P^{i})\cong P^{i}$, one has the exact sequence of projective left $R$-modules
$$\textbf{C}(\Delta): \cdots\longrightarrow P^{-1}\longrightarrow P^{0}\stackrel{\textbf{C}(\varphi^{0})}\longrightarrow P^{1}\longrightarrow P^{2}\longrightarrow \cdots$$ with
$\coker(f)\cong\ker(\textbf{C}(\varphi^{0}))$.

Let $G$ be a projective left $R$-module. Then $\textbf{Z}(G)\cong$ Add$(\textbf{Z}(R))$. Since $pd(_{R\ltimes_{n}M}\textbf{Z}(R))<\infty$, we have $\Hom_{R}(\textbf{C}(\Delta),G)\cong\Hom_{R\ltimes_{n}M}(\Delta, \textbf{Z}(G))$ is exact by \cite[Proposition 2.3]{H}. So $\coker(f)$ is a Gorenstein projective left $R$-module.

Since $\textbf{Z}(R)\otimes_{R\ltimes_{n}M}\Delta$ is exact, one has $\Tor^{R\ltimes_{n}M}_{1}(\textbf{Z}(R), (X,f))=0$. Thus the sequence  $M^{\otimes_{R}n+1}\otimes_{R} X\stackrel{(M\otimes f) \cdots(M^{\otimes_{R}n}\otimes f)}\longrightarrow  M\otimes_{R} X\stackrel{f}\longrightarrow X$  is exact by \cite[Theorem 4.4]{M1}.
\end{proof}
Next, we  establish  sufficient  conditions for Gorenstein projective modules over $n$-trivial extensions.
\begin{thm} \label{thm: 3.2}Assume that   $fd(M^{\otimes_{R}i}_{R})<\infty$ and $pd(_{R}M^{\otimes_{R}i})<\infty$ for any $1\leq i\leq n$. If $(X,f)$ is a left $R\ltimes_{n}M$-module such that the sequence $M^{\otimes_{R}n+1}\otimes_{R} X\stackrel{(M\otimes f) \cdots(M^{\otimes_{R}n}\otimes f)}\longrightarrow  M\otimes_{R} X\stackrel{f}\longrightarrow X$  is exact and $\coker(f)$ is a Gorenstein projective left $R$-module, then  $(X,f)$ is a Gorenstein projective left $R\ltimes_{n}M$-module.
\end{thm}
\begin{proof}There is an exact sequence of projective left $R$-modules $$\Xi: \cdots\longrightarrow
P^{-1}\stackrel{g^{-1}}\longrightarrow P^{0}\stackrel{g^{0}}\longrightarrow P^{1}\stackrel{g^{1}}\longrightarrow P^{2}\longrightarrow \cdots$$  such that $\coker(f)\cong\ker(g^{0})$ and $\Hom_{R}(\Xi, P)$ is also exact for any projective left $R$-module $P$.

Let $f_{1}=f$, $f_{j}=f(M\otimes f) \cdots(M^{\otimes_{R}j-1}\otimes f): M^{\otimes_{R}j}\otimes_{R} X\rightarrow X$ and $\pi_{j}: X\rightarrow\coker(f_{j})$ be the canonical epimorphism. Then $f_{n+1}=0$ and $\pi_{n+1}$ is just the identity map.

Since the sequence $M^{\otimes_{R}n+1}\otimes_{R} X\stackrel{M\otimes f_{n}}\longrightarrow  M\otimes_{R} X\stackrel{f}\longrightarrow X$  is exact, it is easy to check that all the sequences $$M^{\otimes_{R}j}\otimes_{R} X\stackrel{M\otimes f_{j-1}}\longrightarrow  M\otimes_{R} X\stackrel{\pi_{j}f}\longrightarrow \coker(f_{j})$$  are exact for $2\leq j\leq n+1$.

Define $\rho_{ji}: \coker(f_{j})\longrightarrow \coker(f_{i})$ by $\rho_{ji}(\pi_{j}(x))=\pi_{i}(x)$ for $i\leq j-1$.  By \cite[Lemma 2.7(1)]{M1}, there are the exact sequences   $$0\longrightarrow M\otimes_{R}\coker(f_{j-1})\stackrel{\alpha_{j}}\longrightarrow \coker(f_{j})\stackrel{\rho_{j1}}\longrightarrow \coker(f)\longrightarrow 0$$  such that $\alpha_{j}(M\otimes \pi_{j-1})=\pi_{j} f$ for $2\leq j\leq n+1$.

Let $\psi_{1}^{0}: \coker(f)\longrightarrow P^{0}$ be the inclusion and $\eta: P^{-1}\longrightarrow\coker(f)$  the  canonical  epimorphism such that $\psi_{1}^{0}\eta=g^{-1}$. There is $\delta: P^{-1}\longrightarrow X$ such that $\pi_{1}\delta=\eta$.

Since $pd(_{R}M^{\otimes_{R}i})<\infty$ for any $1\leq i\leq n$, $pd(_{R}M^{\otimes_{R}i}\otimes_{R}P^{k})<\infty$ for any $k\in\mathbb{N}$. Thus the complex $\Hom_{R}(\Xi, M^{\otimes_{R}i}\otimes_{R}P^{k})$ is exact. So $\Ext^{1}_{R}(\ker(g^{k}), M^{\otimes_{R}i}\otimes_{R}P^{k})=0$. In particular,  $\Ext^{1}_{R}(\coker(f), M^{\otimes_{R}i}\otimes_{R}P^{0})=0$.

We next define $\psi_{j}^{i}: \coker(f_{j})\longrightarrow M^{\otimes_{R}i}\otimes_{R}P^{0}$ by induction for $2\leq j\leq n+1$ and $0\leq i\leq j-1$.

Write $\psi_{j}^{0}=\psi_{1}^{0}\rho_{j1}: \coker(f_{j})\longrightarrow P^{0}$.

There is $\psi_{j}^{j-1}: \coker(f_{j})\longrightarrow M^{\otimes_{R}j-1}\otimes_{R}P^{0}$ such that $\psi_{j}^{j-1}\alpha_{j}=M\otimes\psi_{j-1}^{j-2}.$

Define $\psi_{j}^{i-1}=\psi_{i}^{i-1}\rho_{ji}: \coker(f_{j})\longrightarrow  M^{\otimes_{R}i-1}\otimes_{R}P^{0}$ for $2\leq i\leq j-1$.

Define $\lambda_{j}: \coker(f_{j})\longrightarrow \coprod_{i=0}^{j-1}(M^{\otimes_{R}i}\otimes_{R} P^{0})$  by $$\lambda_{j}(y)=(\psi_{j}^{0}(y),\psi_{j}^{1}(y), \psi_{j}^{2}(y),\cdots,\psi_{j}^{j-1}(y)).$$

Define $\xi_{j}: \coprod_{i=0}^{j-1}(M^{\otimes_{R}i}\otimes_{R} P^{-1})\longrightarrow \coker(f_{j})$ by $$\xi_{j}(y_{0},y_{1},\cdots,y_{j-1})=\pi_{j}\delta(y_{0})+\alpha_{j}(M\otimes \pi_{j-1}\delta)(y_{1})+\alpha_{j}(M\otimes \alpha_{j-1}(M\otimes\pi_{j-2}\delta))(y_{2})+\cdots$$$$+\alpha_{j}(M\otimes \alpha_{j-1}(M\otimes\alpha_{j-2}\cdots(M\otimes \alpha_{2}(M\otimes\pi_{1}\delta))\cdots))(y_{j-1}).$$

Since  $fd(M_{R})<\infty$, one has the exact sequence of  left $R$-modules
$$M\otimes_{R}\Xi: \cdots\longrightarrow
M\otimes_{R} P^{-1}\stackrel{M\otimes g^{-1}}\longrightarrow
M\otimes_{R} P^{0}\stackrel{M\otimes g^{0}}\longrightarrow
M\otimes_{R} P^{1}\stackrel{M\otimes g^{1}}\longrightarrow
M\otimes_{R} P^{2}\longrightarrow \cdots$$ with
$M\otimes_{R}\coker(f)\cong\ker({M\otimes g^{0}})$.

By  the generalized Horseshoe Lemma \cite[Lemma 1.6]{Z}, one gets the following commutative diagrams:
$$\xymatrix{&0\ar[d]&&0\ar[d] &0\ar[d] &\\
0 \ar[r] &M\otimes_R \coker(f)\ar[d]^{M\otimes\psi_{1}^{0}}\ar[rr]^{\alpha_{2}}&&\coker(f_{2})\ar@{.>}[d]^{\lambda_{2}}\ar[r]^{\rho_{21}}&\coker(f)\ar[r]\ar[d]^{\psi_{1}^{0}}&0\\
0\ar[r] &M\otimes_{R}P^{0}\ar[d]^{^{M\otimes g^{0}}}\ar[rr]&&P^{0}\oplus (M\otimes_{R}P^{0})\ar@{.>}[d]^{h_{2}^{0}}
\ar[r]&P^{0}\ar[r]\ar[d]^{g^{0}}&0\\
0\ar[r]&M\otimes_{R}P^{1}\ar[d]^{M\otimes g^{1}}\ar[rr]&&P^{1}\oplus (M\otimes_{R}P^{1})\ar@{.>}[d]^{h_{2}^{1}}
\ar[r]&P^{1}\ar[r]\ar[d]^{g^{1}}&0\\
0\ar[r]&M\otimes_{R} P^{2}\ar[d]\ar[rr]&&P^{2}\oplus (M\otimes_{R}P^{2})
\ar@{.>}[d]\ar[r]&P^{2}\ar[d]\ar[r]&0\\
&\vdots &&\vdots &\vdots &}$$
and
$$\xymatrix{&\vdots\ar[d]
&&\vdots\ar@{.>}[d] &\vdots\ar[d] &\\
0\ar[r] &M\otimes_{R}P^{-2}\ar[d]^{^{M\otimes g^{-2}}}\ar[rr]&&P^{-2}\oplus (M\otimes_{R}P^{-2})\ar@{.>}[d]^{h_{2}^{-2}}
\ar[r]&P^{-2}\ar[r]\ar[d]^{g^{-2}}&0\\0
\ar[r]&M\otimes_{R}P^{-1}\ar[d]^{M\otimes \eta}\ar[rr]&&P^{-1}\oplus (M\otimes_{R}P^{-1})\ar@{.>}[d]^{\xi_{2}}
\ar[r]&P^{-1}\ar[r]\ar[d]^{\eta}&0\\0\ar[r] &M\otimes_R \coker(f)\ar[d]\ar[rr]^{\alpha_{2}}&&\coker(f_{2})\ar[d]
\ar[r]^{\rho_{21}}&\coker(f)\ar[r]\ar[d]&0\\
&0&&0 &0 &}$$ with exact rows and columns and $h_{2}^{k}=\biggl(\begin{matrix} g^{k}&0\\\sigma_{2}^{k}&M\otimes g^{k}\end{matrix}\biggr)$.

Since  $fd(M\otimes_{R}M_{R})<\infty$, the functor $M\otimes_{R}-$ leaves the middle columns in the above commutative diagrams exact. By  the generalized Horseshoe Lemma, the exact sequence  $$0\longrightarrow M\otimes_{R}\coker(f_{2})\stackrel{\alpha_{3}}\longrightarrow \coker(f_{3})\stackrel{\rho_{31}}\longrightarrow \coker(f)\longrightarrow 0$$
yields the exact sequence of  left $R$-modules
$$0\rightarrow \coker(f_{3})\stackrel{\lambda_{3}}\longrightarrow  \coprod_{i=0}^{2}(M^{\otimes_{R}i}\otimes_{R} P^{0})\stackrel{h_{3}^{0}}\longrightarrow \coprod_{i=0}^{2}(M^{\otimes_{R}i}\otimes_{R} P^{1})\stackrel{h_{3}^{1}}\longrightarrow  \coprod_{i=0}^{2}(M^{\otimes_{R}i}\otimes_{R} P^{2})\longrightarrow \cdots$$   and $$\cdots\longrightarrow  \coprod_{i=0}^{2}(M^{\otimes_{R}i}\otimes_{R} P^{-2})\stackrel{h_{3}^{-2}}\longrightarrow \coprod_{i=0}^{2}(M^{\otimes_{R}i}\otimes_{R} P^{-1})\stackrel{\xi_{3}}\longrightarrow  \coker(f_{3})\rightarrow 0$$ with $h_{3}^{k}=\biggl(\begin{matrix} g^{k}&0&0\\\sigma^{k}_{2}&M\otimes g^{k}&0\\\sigma^{k}_{3}&M\otimes\sigma^{k}_{2}&M\otimes_{R} M\otimes g^{k}\end{matrix}\biggr)$.

Iteratively,  the exact sequence $$0\longrightarrow M\otimes_{R}\coker(f_{n})\stackrel{\alpha_{n+1}}\longrightarrow X\stackrel{\pi_{1}}\longrightarrow \coker(f)\longrightarrow 0$$  yields the exact sequence of  left $R$-modules $$0\rightarrow X\stackrel{\lambda_{n+1}}\longrightarrow  \coprod_{i=0}^{n}(M^{\otimes_{R}i}\otimes_{R} P^{0})\stackrel{h_{n+1}^{0}}\longrightarrow \coprod_{i=0}^{n}(M^{\otimes_{R}i}\otimes_{R} P^{1})\stackrel{h_{n+1}^{1}}\longrightarrow  \coprod_{i=0}^{n}(M^{\otimes_{R}i}\otimes_{R} P^{2})\longrightarrow \cdots$$   and $$\cdots\longrightarrow  \coprod_{i=0}^{n}(M^{\otimes_{R}i}\otimes_{R} P^{-2})\stackrel{h_{n+1}^{-2}}\longrightarrow \coprod_{i=0}^{n}(M^{\otimes_{R}i}\otimes_{R} P^{-1})\stackrel{\xi_{n+1}}\longrightarrow X\rightarrow 0$$ with $h_{n+1}^{k}=\begin{pmatrix} g^{k}&0&\cdots&0\\\sigma^{k}_{2}&M\otimes g^{k}&\cdots&0\\\vdots&\vdots&&\vdots\\\sigma^{k}_{n+1}&M\otimes\sigma^{k}_{n}&\cdots&M^{\otimes_{R}n}\otimes g^{k}\end{pmatrix}$.

The following commutative diagram
$${\Small \xymatrix{&M\otimes_R X\ar[d]_{f}\ar[rr]^{M\otimes\lambda_{n+1}}&&M\otimes_R\coprod_{i=0}^{n}(M^{\otimes_{R}i}\otimes_{R} P^{0})\ar[d]^{\mu_{0}}\ar[r]^{M\otimes h_{n+1}^{0}}&M\otimes_R\coprod_{i=0}^{n}(M^{\otimes_{R}i}\otimes_{R} P^{1})\ar[d]^{\mu_{1}}\ar[r]&\cdots\\
0\ar[r]&X\ar[rr]^{\lambda_{n+1}}&& \coprod_{i=0}^{n}(M^{\otimes_{R}i}\otimes_{R} P^{0})\ar[r]^{h_{n+1}^{0}}& \coprod_{i=0}^{n}(M^{\otimes_{R}i}\otimes_{R} P^{1})\ar[r]&\cdots}}$$ implies that the sequence $0\longrightarrow(X,f)\stackrel{\lambda_{n+1}}\longrightarrow\textbf{T}(P^{0})\stackrel{h_{n+1}^{0}}\longrightarrow \textbf{T}(P^{1})\longrightarrow\cdots$ is exact.

The following commutative diagram
$${\Small\xymatrix{\cdots\ar[r]&M\otimes_R \coprod_{i=0}^{n}(M^{\otimes_{R}i}\otimes_{R} P^{-2})\ar[d]^{\mu_{-2}}\ar[r]^{M\otimes h_{n+1}^{-2}}&M\otimes_R \coprod_{i=0}^{n}(M^{\otimes_{R}i}\otimes_{R} P^{-1})\ar[d]^{\mu_{-1}}\ar[rr]^{M\otimes\xi_{n+1}}&&M\otimes_R X\ar[d]_{f}\\
\cdots\ar[r]& \coprod_{i=0}^{n}(M^{\otimes_{R}i}\otimes_{R} P^{-2})\ar[r]^{h_{n+1}^{-2}}& \coprod_{i=0}^{n}(M^{\otimes_{R}i}\otimes_{R} P^{-1})\ar[rr]^{\xi_{n+1}}&&X\ar[r]&0}}$$ implies that the sequence $\cdots\longrightarrow\textbf{T}(P^{-2})\stackrel{h_{n+1}^{-2}}\longrightarrow \textbf{T}(P^{-1})\stackrel{\xi_{n+1}}\longrightarrow(X,f)\longrightarrow 0$ is exact.

Let $h_{n+1}^{-1}=\lambda_{n+1}\xi_{n+1}$. Then we obtain the desired exact sequence of projective left $R\ltimes_{n}M$-modules $$\Delta: \cdots\longrightarrow \textbf{T}(P^{-2})\stackrel{h_{n+1}^{-2}}\longrightarrow \textbf{T}(P^{-1})\stackrel{h_{n+1}^{-1}}\longrightarrow
\textbf{T}(P^{0})\stackrel{h_{n+1}^{0}}\longrightarrow \textbf{T}(P^{1})\stackrel{h_{n+1}^{1}}\longrightarrow \textbf{T}(P^{2})\longrightarrow \cdots$$ with $(X,f)\cong\ker(h_{n+1}^{0})$.

By \cite[Theorem 2.8]{M1}, a projective left $R\ltimes_{n}M$-module is isomorphic to $\textbf{T}(Q)$ with $Q$  a projective left $R$-module. By \cite[Lemma 2.6(3)]{M1}, there are exact sequences  $$0\rightarrow (\coprod_{i=1}^{n}(M^{\otimes_{R}i}\otimes_{R}Q), w_{1})\rightarrow \textbf{T}(Q)\rightarrow \textbf{Z}(Q)\rightarrow 0,$$  $$0\rightarrow(\coprod_{i=2}^{n}(M^{\otimes_{R}i}\otimes_{R}Q), w_{2})\rightarrow (\coprod_{i=1}^{n}(M^{\otimes_{R}i}\otimes_{R}Q), w_{1})\rightarrow\textbf{Z}(M\otimes_{R} Q)\rightarrow 0,$$ $$\cdots\cdots$$  $$0\rightarrow\textbf{Z}(M^{\otimes_{R}n}\otimes_{R}Q)\rightarrow (\coprod_{i=n-1}^{n}(M^{\otimes_{R}i}\otimes_{R}Q), w_{n-1})\rightarrow\textbf{Z}(M^{\otimes_{R}n-1}\otimes_{R} Q)\rightarrow 0.$$
 Hence one gets the exact sequences of complexes {\small $$0\rightarrow \Hom_{R\ltimes_{n}M}(\Delta,(\coprod_{i=1}^{n}(M^{\otimes_{R}i}\otimes_{R}Q), w_{1}))\rightarrow \Hom_{R\ltimes_{n}M}(\Delta,\textbf{T}(Q))\rightarrow \Hom_{R\ltimes_{n}M}(\Delta,\textbf{Z}(Q))\rightarrow 0,$$}
{\tiny $$0\rightarrow \Hom_{R\ltimes_{n}M}(\Delta,(\coprod_{i=2}^{n}(M^{\otimes_{R}i}\otimes_{R}Q), w_{2}))\rightarrow \Hom_{R\ltimes_{n}M}(\Delta,(\coprod_{i=1}^{n}(M^{\otimes_{R}i}\otimes_{R}Q), w_{1}))\rightarrow \Hom_{R\ltimes_{n}M}(\Delta,\textbf{Z}(M\otimes_{R} Q))\rightarrow 0,$$
$$\cdots\cdots$$  $$0\rightarrow \Hom_{R\ltimes_{n}M}(\Delta,\textbf{Z}(M^{\otimes_{R}n}\otimes_{R}Q))\rightarrow \Hom_{R\ltimes_{n}M}(\Delta, (\coprod_{i=n-1}^{n}(M^{\otimes_{R}i}\otimes_{R}Q), w_{n-1}))\rightarrow \Hom_{R\ltimes_{n}M}(\Delta,\textbf{Z}(M^{\otimes_{R}n-1}\otimes_{R}Q))\rightarrow 0.$$}
Since $pd(_{R}M^{\otimes_{R}i})<\infty$, one has $pd(_{R}M^{\otimes_{R}i}\otimes_{R}Q)<\infty$. This implies that  $\Hom_{R}(\Xi,M^{\otimes_{R}i}\otimes_{R}Q)$ is exact. So $\Hom_{R\ltimes_{n}M}(\Delta,\textbf{Z}(M^{\otimes_{R}i}\otimes_{R}Q))\cong\Hom_{R}(\Xi,M^{\otimes_{R}i}\otimes_{R}Q)$ is exact for any $0\leq i\leq n$. Thus $\Hom_{R\ltimes_{n}M}(\Delta,\textbf{T}(Q))$ is exact by \cite[Theorem 1.4.7]{EJ}.

It follows that  $(X,f)$ is a Gorenstein projective left $R\ltimes_{n}M$-module.
\end{proof}
Combining Theorems \ref{thm: 3.1} and \ref{thm: 3.2}, we have
\begin{cor} \label{cor: 3.3}Assume that  $fd(M^{\otimes_{R}i}_{R})<\infty$ and $pd(_{R}M^{\otimes_{R}i})<\infty$ for any $1\leq i\leq n$, $fd(\textbf{Z}(R)_{R\ltimes_{n}M})<\infty$ and $pd(_{R\ltimes_{n}M}\textbf{Z}(R))<\infty$. Then
\begin{enumerate}\item $\textbf{T}(X)$ is a Gorenstein projective left $R\ltimes_{n}M$-module if and only if  $X$ is a Gorenstein projective left $R$-module.\item $\textbf{Z}(X)$ is a Gorenstein projective left $R\ltimes_{n}M$-module if and only if $M\otimes_R X=0$ and $X$ is a Gorenstein projective left $R$-module.
\end{enumerate}
\end{cor}
\begin{lem}\label{lem: 3.4}Let $[Y,g]$ be a left $R\ltimes_{n}M$-module. Then there are exact sequences in $R\ltimes_{n}M$-Mod  $$0\longrightarrow \textbf{Z}(\ker(g))\longrightarrow [Y,g]\longrightarrow[\im(g),w_{1}]\longrightarrow 0,$$  $$0\rightarrow\textbf{Z}(\ker(w_{1}))\rightarrow [\im(g),w_{1}] \rightarrow[\im(w_{1}),w_{2}]\rightarrow 0,$$ $$\cdots\cdots$$ $$0\rightarrow\textbf{Z}(\ker(w_{n-2}))\rightarrow[\im(w_{n-3}),w_{n-2}]\rightarrow[\im(w_{n-2}),w_{n-1}]\rightarrow 0,$$
 $$0\rightarrow\textbf{Z}(\ker(w_{n-1}))\rightarrow[\im(w_{n-2}),w_{n-1}]\rightarrow\textbf{Z}(\ker(w_{n}))\rightarrow 0.$$
\end{lem}
\begin{proof}Let $g_{i}=\Hom_{R}^{i-1}(M,g)\cdots\Hom_{R}(M,g)g: Y\rightarrow \Hom_{R}^{i}(M,Y)$.

There exists the following commutative diagram with exact rows: $$\xymatrix{0\ar[r]&\ker(g)\ar[r]\ar[d]_{0}&Y\ar[d]_{g}\ar[r]&\im(g)\ar[d]_{w_{1}}\ar[r]&0\\0\ar[r]&\Hom_{R}(M,\ker(g))\ar[r]&\Hom_{R}(M,Y)\ar[r]&\Hom_{R}(M,\im(g))}$$ where $w_{1}$ is induced by $g$ such that $\ker(w_{1})= g(\ker(g_{2}))$. We can form this diagram again with $[Y,g]$ replaced by $[\im(g),w_{1}]$ and get the left $R\ltimes_{n}M$-module $[\im(w_{1}),w_{2}]$ such that $\ker(w_{2})= w_{1}(g(\ker(g_{3})))= g_{2}(\ker(g_{3}))$. The next step gives us the left $R\ltimes_{n}M$-module $[\im(w_{2}),w_{3}]$ such that $\ker(w_{3})= w_{2}(w_{1}(g(\ker(g_{4}))))= g_{3}(\ker(g_{4}))$. Since $g_{n+1}=0$, we have $w_{n}=0$. So we will eventually reach the above commutative diagram with the two extreme homomorphisms equal to zero by this process, i.e., there exist  exact sequences $$0\longrightarrow \textbf{Z}(\ker(g))\longrightarrow [Y,g]\longrightarrow[\im(g),w_{1}]\longrightarrow 0,$$  $$0\rightarrow\textbf{Z}(\ker(w_{1}))\rightarrow [\im(g),w_{1}] \rightarrow[\im(w_{1}),w_{2}]\rightarrow 0,$$ $$\cdots\cdots$$ $$0\rightarrow\textbf{Z}(\ker(w_{n-2}))\rightarrow[\im(w_{n-3}),w_{n-2}]\rightarrow[\im(w_{n-2}),w_{n-1}]\rightarrow 0,$$
 $$0\rightarrow\textbf{Z}(\ker(w_{n-1}))\rightarrow[\im(w_{n-2}),w_{n-1}]\rightarrow\textbf{Z}(\ker(w_{n}))\rightarrow 0.$$

This completes the proof.
\end{proof}
Recall that a left $R$-module $Y$ is \emph{Gorenstein injective} \cite{EJ1} if there is an exact sequence of injective left $R$-modules $\Theta: \cdots\rightarrow
E^{-1}\rightarrow E^{0}\stackrel{f^{0}}\rightarrow E^{1}\rightarrow
E^{2}\rightarrow \cdots$ such that $Y\cong\ker(f^{0})$ and $\Hom_{R}(G,\Theta)$ is also exact for any injective left $R$-module $G$.

Using Lemma \ref{lem: 3.4} and a dual argument of the proof of Theorems \ref{thm: 3.1} and \ref{thm: 3.2}, we have
\begin{thm} \label{thm: 3.5}Assume that  $[Y,g]$ is a left $R\ltimes_{n}M$-module.
\begin{enumerate}\item If $fd(\textbf{Z}(R)_{R\ltimes_{n}M})<\infty$, $pd(_{R\ltimes_{n}M}\textbf{Z}(R))$ or $id(_{R\ltimes_{n}M}\textbf{Z}(R))<\infty$, and $[Y,g]$ is a Gorenstein injective left $R\ltimes_{n}M$-module, then the sequence $Y\stackrel{g}\longrightarrow \Hom_{R}(M,Y)$ $\stackrel{{\rm Hom}_{R}^{n}(M,g)\cdots{\rm Hom}_{R}(M,g)}\longrightarrow \Hom^{n+1}_{R}(M,Y)$  is exact and $\ker(g)$ is a Gorenstein injective left $R$-module.\item If $fd(M^{\otimes_{R}i}_{R})<\infty$, $pd(_{R}M^{\otimes_{R}i})<\infty$ for any $1\leq i\leq n$, the sequence $Y\stackrel{g}\longrightarrow \Hom_{R}(M,Y)\stackrel{{\rm Hom}_{R}^{n}(M,g)\cdots{\rm Hom}_{R}(M,g)}\longrightarrow \Hom^{n+1}_{R}(M,Y)$  is exact and $\ker(g)$ is a Gorenstein injective left $R$-module, then $[Y,g]$  is a Gorenstein injective left $R\ltimes_{n}M$-module.
\end{enumerate}
\end{thm}
\begin{cor} \label{cor: 3.6}Assume that  $fd(M^{\otimes_{R}i}_{R})<\infty$, $pd(_{R}M^{\otimes_{R}i})<\infty$ for any $1\leq i\leq n$, $fd(\textbf{Z}(R)_{R\ltimes_{n}M})<\infty$, $pd(_{R\ltimes_{n}M}\textbf{Z}(R))$ or $id(_{R\ltimes_{n}M}\textbf{Z}(R))<\infty$. Then
\begin{enumerate}\item $\textbf{H}(Y)$  is a Gorenstein injective left $R\ltimes_{n}M$-module if and only if  $Y$ is a Gorenstein injective left $R$-module.\item $\textbf{Z}(Y)$  is a Gorenstein injective left $R\ltimes_{n}M$-module if and only if $\Hom_R(M, Y)=0$ and $Y$ is a Gorenstein injective left $R$-module.
\end{enumerate}
\end{cor}
Recall that a \emph{tensor ring} of an $R$-bimodule $M$ is $T_{R}(M) = \coprod_{i=0}^{\infty}M^{\otimes_{R}i}$ \cite{Co}. It is easy to verify that there is a ring isomorphism $R\ltimes_{n} M\cong T_{R}(M)/\coprod_{i=n+1}^{\infty}M^{\otimes_{R}i}$. In particular, if  $M$ is $n$-nilpotent, i.e., $M^{\otimes_{R}n+1}=0$, then $R\ltimes_{n} M\cong T_{R}(M)$. Recently, Gorenstein homological modules over the tensor ring $T_{R}(M)$ with $M$ nilpotent have been investigated by many authors \cite{CL, DLST, TW}.

As consequences of Theorems \ref{thm: 3.1}, \ref{thm: 3.2} and \ref{thm: 3.5}, we have
\begin{cor} \label{cor: 3.7} Let $M$ be an $n$-nilpotent $R$-bimodule and $(X,f)$ a left $T_{R}(M)$-module.
\begin{enumerate}\item If $fd(\textbf{Z}(R)_{T_{R}(M)})<\infty$, $pd(_{T_{R}(M)}\textbf{Z}(R))<\infty$ and $(X,f)$ is a Gorenstein projective left $T_{R}(M)$-module, then $f$ is a monomorphism and $\coker(f)$ is a Gorenstein projective left $R$-module.\item If $fd(M^{\otimes_{R}i}_{R})<\infty$ and $pd(_{R}M^{\otimes_{R}i})<\infty$ for any $1\leq i\leq n$, $f$ is a monomorphism and $\coker(f)$ is a Gorenstein projective left $R$-module, then  $(X,f)$ is a Gorenstein projective left $T_{R}(M)$-module.
\end{enumerate}
\end{cor}
\begin{cor} \label{cor: 3.8}Let $M$ be an $n$-nilpotent $R$-bimodule and $[Y,g]$ a left $T_{R}(M)$-module.
\begin{enumerate}\item If $fd(\textbf{Z}(R)_{T_{R}(M)})<\infty$, $pd(_{T_{R}(M)}\textbf{Z}(R))$ or $id(_{T_{R}(M)}\textbf{Z}(R))<\infty$, $[Y,g]$ is a Gorenstein injective left $T_{R}(M)$-module, then $g$ is an epimorphism and $\ker(g)$ is a Gorenstein injective left $R$-module.\item If $fd(M^{\otimes_{R}i}_{R})<\infty$, $pd(_{R}M^{\otimes_{R}i})<\infty$ for any $1\leq i\leq n$,  $g$ is an epimorphism and $\ker(g)$ is a Gorenstein injective left $R$-module, then $[Y,g]$  is a Gorenstein injective left $T_{R}(M)$-module.
\end{enumerate}
\end{cor}
\begin{rem} \label{rem: 3.9} {\rm  Let $M$ be an $n$-nilpotent $R$-bimodule such that $\Tor^{R}_{j}(M, M^{\otimes_{R}i})=0$ for any $i,j\geq 1$, $fd(M_{R})<\infty$ and $pd(_{R}M)<\infty$.

By \cite[Lemma 4.5]{CL}, $fd(M^{\otimes_{R}i}_{R})<\infty$ and $pd(_{R}M^{\otimes_{R}i})<\infty$ for any $1\leq i\leq n$. By \cite[Corollary 1.10(2)]{DLST} and \cite[Corollary 4.3]{M1}, $pd(_{T_{R}(M)}\textbf{T}(M))<\infty$ and $fd(\textbf{T}(M)_{T_{R}(M)})<\infty$.

Also, there is the exact sequence $0\rightarrow\textbf{T}(M)\rightarrow\textbf{T}(R)\rightarrow\textbf{Z}(R)\rightarrow 0$. So  $fd(\textbf{Z}(R)_{T_{R}(M)})<\infty$ and $pd(_{T_{R}(M)}\textbf{Z}(R))<\infty$. Thus \cite[Corollary  2.13]{DLST} is an immediate consequence of Corollary \ref{cor: 3.7}, \cite[Corollary  2.20]{DLST} is an immediate consequence of Corollary \ref{cor: 3.8}.}
\end{rem}
\bigskip
\section {\bf Gorenstein flat and $PGF$-modules over $n$-trivial extensions}
\bigskip
Recall that a left $R$-module $X$ is \emph{Gorenstein flat} \cite{EJT} if there is an exact sequence $\cdots\longrightarrow F^{-1}\longrightarrow
F^{0}\stackrel{f^{0}}\longrightarrow F^{1}\longrightarrow F^{2}\longrightarrow \cdots$  of flat left $R$-modules with $X\cong\ker(f^{0})$, which remains exact after applying $E\otimes_{R}-$ for any injective right $R$-module $E$.
\begin{thm} \label{thm: 4.1}Assume that   $fd(_{R\ltimes_{n}M}\textbf{Z}(R))<\infty$,  $fd(\textbf{Z}(R)_{R\ltimes_{n}M})$ or $id(\textbf{Z}(R)_{R\ltimes_{n}M})<\infty$. If $(X,f)$ is a Gorenstein flat left $R\ltimes_{n}M$-module, then the sequence  $M^{\otimes_{R}n+1}\otimes_{R} X\stackrel{(M\otimes f) \cdots(M^{\otimes_{R}n}\otimes f)}\longrightarrow  M\otimes_{R} X\stackrel{f}\longrightarrow X$   is exact and $\coker(f)$ is a Gorenstein flat left $R$-module.
\end{thm}
\begin{proof}There is an exact sequence of flat left  $R\ltimes_{n}M$-modules $$\Delta: \cdots\longrightarrow(F^{-1},h^{-1})\longrightarrow
(F^{0},h^{0})\stackrel{\varphi^{0}}\longrightarrow (F^{1},h^{1})\longrightarrow (F^{2},h^{2})\longrightarrow \cdots$$
with $(X,f)\cong\ker(\varphi^{0})$, which remains exact after applying $E\otimes_{R\ltimes_{n}M}-$ for any injective right $R\ltimes_{n}M$-module $E$.

Since   $fd(\textbf{Z}(R)_{R\ltimes_{n}M})$ or $id(\textbf{Z}(R)_{R\ltimes_{n}M})<\infty$, $\textbf{Z}(R)\otimes_{R\ltimes_{n}M}\Delta$ is exact. Note that $\textbf{Z}(R)\otimes_{R\ltimes_{n}M} (F^{i},h^{i})\cong  \coker(h^{i})$ is flat by \cite[Theorem 2.10]{M1}. Hence
we get the exact sequence of flat left $R$-modules
$$\textbf{C}(\Delta): \cdots\longrightarrow \coker(h^{-1})\longrightarrow \coker(h^{0})\stackrel{\textbf{C}(\varphi^{0})}\longrightarrow \coker(h^{1})\longrightarrow \coker(h^{2})\longrightarrow \cdots$$ with
$\coker(f)\cong\ker(\textbf{C}(\varphi^{0}))$.

Let $Q$ be an injective right $R$-module.  Then $id(\textbf{Z}(Q)_{R\ltimes_{n}M})<\infty$ since $fd(_{R\ltimes_{n}M}\textbf{Z}(R))<\infty$. Thus $Q\otimes_{R}\textbf{C}(\Delta)\cong \textbf{Z}(Q)\otimes_{R\ltimes_{n}M}\Delta$ is exact. This implies that  $\coker(f)$ is a Gorenstein  flat left $R$-module.

Since $\textbf{Z}(R)\otimes_{R\ltimes_{n}M}\Delta$ is exact,  $\Tor^{R\ltimes_{n}M}_{1}(\textbf{Z}(R), (X,f))=0$. So the sequence  $M^{\otimes_{R}n+1}\otimes_{R} X\stackrel{(M\otimes f) \cdots(M^{\otimes_{R}n}\otimes f)}\longrightarrow  M\otimes_{R} X\stackrel{f}\longrightarrow X$  is exact by \cite[Theorem 4.4]{M1}.
\end{proof}
Recall that $R$ is  a \emph{right coherent ring} \cite{L} if every finitely generated  right ideal of $R$ is finitely presented.
\begin{thm} \label{thm: 4.2}Let  $R\ltimes_{n}M$ be a right coherent ring,  $fd(_{R}M^{\otimes_{R}i})<\infty$ and $pd(M^{\otimes_{R}i}_{R})<\infty$ for any $1\leq i\leq n$. If $(X,f)$ is a left $R\ltimes_{n}M$-module such that the sequence  $M^{\otimes_{R}n+1}\otimes_{R} X\stackrel{(M\otimes f) \cdots(M^{\otimes_{R}n}\otimes f)}\longrightarrow  M\otimes_{R} X\stackrel{f}\longrightarrow X$   is exact and $\coker(f)$ is a Gorenstein flat left $R$-module, then $(X,f)$ is a Gorenstein flat left $R\ltimes_{n}M$-module.
\end{thm}
\begin{proof}Note that $R$ is also a right  coherent ring by \cite[Theorem 3.1]{M1}.

We write $\omega: (M\otimes_{R} X)^{+}\rightarrow\Hom_{R}(M, X^{+})$ to be the natural isomorphism. Then the sequence $X^{+}\stackrel{\omega f^{+}}\longrightarrow \Hom_{R}(M,X^{+})\stackrel{{\rm Hom}_{R}^{n}(M,\omega f^{+})\cdots{\rm Hom}_{R}(M,\omega f^{+})}\longrightarrow \Hom^{n+1}_{R}(M,X^{+})$  is   exact and $\ker(\omega f^{+})\cong (\coker(f))^{+}$ is a Gorenstein injective right $R$-module by \cite[Theorem 3.6]{H}. Thus $(X,f)^{+}\cong[X^{+},\omega f^{+}]$  is a Gorenstein injective right $R\ltimes_{n}M$-module by Theorem \ref{thm: 3.5}. So $(X,f)$ is a Gorenstein flat left $R\ltimes_{n}M$-module by \cite[Theorem 3.6]{H}.
\end{proof}
\begin{cor}\label{cor: 4.3}Assume that $fd(_{R}M^{\otimes_{R}i})<\infty$, $fd(M^{\otimes_{R}i}_{R})<\infty$ for any $1\leq i\leq n$, $fd(_{R\ltimes_{n}M}\textbf{Z}(R))<\infty$, $fd(\textbf{Z}(R)_{R\ltimes_{n}M})$ or $id(\textbf{Z}(R)_{R\ltimes_{n}M})<\infty$. Then
\begin{enumerate}\item $\textbf{T}(X)$ is a Gorenstein flat left $R\ltimes_{n}M$-module if and only if  $X$ is a Gorenstein flat left $R$-module.\item $\textbf{Z}(X)$ is a Gorenstein flat left $R\ltimes_{n}M$-module if and only if $M\otimes_R X =0$ and $X$ is a Gorenstein flat left $R$-module.
\end{enumerate}
\end{cor}
\begin{proof}(1) $``\Rightarrow"$ It follows from Theorem \ref{thm: 4.1}.

$``\Leftarrow"$ There is an exact sequence of flat left $R$-modules $$\Xi: \cdots\longrightarrow
F^{-1}\longrightarrow F^{0}\stackrel{g^{0}}\longrightarrow F^{1}\longrightarrow F^{2}\longrightarrow \cdots$$  such that $X\cong\ker(g^{0})$ and $E\otimes_{R}-$ leaves the sequence exact for any injective right $R$-module $E$. Since $fd(M^{\otimes_{R}i}_{R})<\infty$ for any $1\leq i\leq n$, we get the exact sequence of flat left  $R\ltimes_{n}M$-modules $$\Delta: \cdots\longrightarrow \textbf{T}(F^{-1})\longrightarrow
\textbf{T}(F^{0})\stackrel{\textbf{T}(g^{0})}\longrightarrow \textbf{T}(F^{1})\longrightarrow \textbf{T}(F^{2})\longrightarrow \cdots$$
with $\textbf{T}(X)\cong\ker(\textbf{T}(g^{0}))$.

By \cite[Theorem 2.9]{M1}, an injective right $R\ltimes_{n}M$-module is isomorphic to $\textbf{H}(Q)$ with $Q$  an injective right $R$-module. Since $fd(_{R}M^{\otimes_{R}i})<\infty$ for any $1\leq i\leq n$, $id(\Hom_{R}^{i}(M,Q)_{R})<\infty$. Thus $\textbf{H}(Q)\otimes_{R\ltimes_{n}M}\Delta\cong\coprod_{i=0}^{n}\Hom_{R}^{i}(M,Q)\otimes_{R}\Xi$ is exact by \cite[Lemma 4.1(1)]{M1}. So $\textbf{T}(X)$ is a Gorenstein flat left $R\ltimes_{n}M$-module.

(2) follows from (1) and Theorem \ref{thm: 4.1}.
\end{proof}
Recall that a left $R$-module $X$ is a \emph{projectively coresolved Gorenstein flat module} (\emph{$PGF$ module}, for short) \cite{SS} if there is an exact sequence $\cdots\longrightarrow P^{-1}\longrightarrow P^{0}\stackrel{h^{0}}\longrightarrow P^{1}\longrightarrow P^{2}\longrightarrow\cdots$ of projective left $R$-modules such that $X\cong\ker(h^{0})$ and $E\otimes_{R}-$ leaves the sequence exact for any injective right $R$-module $E$. The $PGF$ modules played a crucial role in \v{S}aroch and \v{S}\'{t}ov\'{\i}\v{c}ek's proof that the cotorsion pair generated by the Gorenstein flat modules is complete. These modules are simultaneously Gorenstein projective and
Gorenstein flat. Thus $PGF$ modules in Gorenstein homological algebra may be viewed as the role of projective modules in classical  homological algebra and have been studied by many authors \cite{E,I,SS}.
\begin{thm} \label{thm: 4.4}Assume that  $fd(_{R\ltimes_{n}M}\textbf{Z}(R))<\infty$,  $fd(\textbf{Z}(R)_{R\ltimes_{n}M})$ or $id(\textbf{Z}(R)_{R\ltimes_{n}M})<\infty$. If $(X,f)$ is a $PGF$ left $R\ltimes_{n}M$-module, then the sequence  $M^{\otimes_{R}n+1}\otimes_{R} X\stackrel{(M\otimes f) \cdots(M^{\otimes_{R}n}\otimes f)}\longrightarrow  M\otimes_{R} X\stackrel{f}\longrightarrow X$   is exact and $\coker(f)$ is a $PGF$ left $R$-module.
\end{thm}
\begin{proof}There is an exact sequence of projective left  $R\ltimes_{n}M$-modules $$\Delta: \cdots\longrightarrow \textbf{T}(P^{-1})\longrightarrow
\textbf{T}(P^{0})\stackrel{\varphi^{0}}\longrightarrow \textbf{T}(P^{1})\longrightarrow \textbf{T}(P^{2})\longrightarrow \cdots$$
with $(X,f)\cong\ker(\varphi^{0})$, which remains exact after applying $E\otimes_{R\ltimes_{n}M}-$ for any injective right $R\ltimes_{n}M$-module $E$.
Since  $fd(\textbf{Z}(R)_{R\ltimes_{n}M})$ or $id(\textbf{Z}(R)_{R\ltimes_{n}M})<\infty$, $\textbf{Z}(R)\otimes_{R\ltimes_{n}M}\Delta$ is exact. Applying the fuctor $\textbf{Z}(R)\otimes_{R\ltimes_{n}M}-$ to the above exact sequence $\Delta$, we get the exact sequence of projective left $R$-modules
$$\textbf{C}(\Delta): \cdots\longrightarrow P^{-1}\longrightarrow P^{0}\stackrel{\textbf{C}(\varphi^{0})}\longrightarrow P^{1}\longrightarrow P^{2}\longrightarrow \cdots$$ with
$\coker(f)\cong\ker(\textbf{C}(\varphi^{0}))$.

Let $Q$ be an injective right $R$-module.  Then $id(\textbf{Z}(Q)_{R\ltimes_{n}M})<\infty$ since $fd(_{R\ltimes_{n}M}\textbf{Z}(R))<\infty$. Thus $Q\otimes_{R}\textbf{C}(\Delta)\cong \textbf{Z}(Q)\otimes_{R\ltimes_{n}M}\Delta$ is exact. So $\coker(f)$ is a $PGF$ left $R$-module.

Since $\textbf{Z}(R)\otimes_{R\ltimes_{n}M}\Delta$ is exact,  $\Tor^{R\ltimes_{n}M}_{1}(\textbf{Z}(R), (X,f))=0$. So the sequence  $M^{\otimes_{R}n+1}\otimes_{R} X\stackrel{(M\otimes f) \cdots(M^{\otimes_{R}n}\otimes f)}\longrightarrow  M\otimes_{R} X\stackrel{f}\longrightarrow X$  is exact by \cite[Theorem 4.4]{M1}.
\end{proof}
\begin{thm}\label{thm: 4.5}Assume that  $fd(_{R}M^{\otimes_{R}i})<\infty$ and $fd(M^{\otimes_{R}i}_{R})<\infty$ for any $1\leq i\leq n$. If $(X,f)$ is a left $R\ltimes_{n}M$-module such that the sequence  $M^{\otimes_{R}n+1}\otimes_{R} X\stackrel{(M\otimes f) \cdots(M^{\otimes_{R}n}\otimes f)}\longrightarrow  M\otimes_{R} X\stackrel{f}\longrightarrow X$   is exact and $\coker(f)$ is a $PGF$ left $R$-module, then $(X,f)$ is a $PGF$ left $R\ltimes_{n}M$-module.
\end{thm}
\begin{proof}There is an exact sequence of projective left $R$-modules $$\Xi: \cdots\longrightarrow
P^{-1}\stackrel{g^{-1}}\longrightarrow P^{0}\stackrel{g^{0}}\longrightarrow P^{1}\stackrel{g^{1}}\longrightarrow P^{2}\stackrel{g^{2}}\longrightarrow \cdots$$  such that $\coker(f)\cong\ker(g^{0})$ and  $E\otimes_{R}-$ leaves the sequence exact for any injective right $R$-module $E$.

Let $f_{1}=f$, $f_{i}=f(M\otimes f) \cdots(M^{\otimes_{R}i-1}\otimes f): M^{\otimes_{R}i}\otimes_{R} X\rightarrow X$. By \cite[Lemma 2.7(1)]{M1}, there are the exact sequences   $$0\longrightarrow M\otimes_{R}\coker(f_{i-1})\longrightarrow \coker(f_{i})\longrightarrow \coker(f)\longrightarrow 0.$$

Since $fd(_{R}M^{\otimes_{R}i})<\infty$ for any $1\leq i\leq n$, $fd(_{R}M^{\otimes_{R}i}\otimes_{R}P^{k})<\infty$ for any $k\in\mathbb{N}$. By \cite[Corollary 1]{I} and \cite[Lemma 3.1]{M0},  we have $\Ext^{1}_{R}(\ker(g^{k}), M^{\otimes_{R}i}\otimes_{R}P^{k})=0$.
Since $fd(M^{\otimes_{R}i}_{R})<\infty$ for any $1\leq i\leq n$, $M^{\otimes_{R}i}\otimes_{R}\Xi$ is exact. Using a similar method as in Theorem \ref{thm: 3.2}, we obtain the desired exact sequence of projective left $R\ltimes_{n}M$-modules $$\Delta: \cdots\longrightarrow \textbf{T}(P^{-1})\stackrel{h_{n+1}^{-1}}\longrightarrow
\textbf{T}(P^{0})\stackrel{h_{n+1}^{0}}\longrightarrow \textbf{T}(P^{1})\stackrel{h_{n+1}^{1}}\longrightarrow \textbf{T}(P^{2})\longrightarrow \cdots$$ with $(X,f)\cong\ker(h_{n+1}^{0})$ by  the generalized Horseshoe Lemma.

Next, we show that the sequence $\Delta$ remains exact when tensored with any injective right $R\ltimes_{\Phi}M$-module  $\textbf{H}(Q)$, where $Q$ is an injective right $R$-module. By Lemma \ref{lem: 3.4}, there are exact sequences in Mod-$R\ltimes_{n}M$  $$0\longrightarrow \textbf{Z}(Q)\longrightarrow \textbf{H}(Q)\longrightarrow[\coprod_{i=1}^{n}\Hom_{R}^{i}(M,Q),w_{1}]\longrightarrow 0,$$  $$0\rightarrow\textbf{Z}(\Hom_{R}(M,Q))\rightarrow [\coprod_{i=1}^{n}\Hom_{R}^{i}(M,Q),w_{1}] \rightarrow[\coprod_{i=2}^{n}\Hom_{R}^{i}(M,Q),w_{2}]\rightarrow 0,$$ $$\cdots\cdots$$
 $$0\rightarrow\textbf{Z}(\Hom_{R}^{n-1}(M,Q))\rightarrow[\coprod_{i=n-1}^{n}\Hom_{R}^{i}(M,Q),w_{n-1}]\rightarrow\textbf{Z}(\Hom_{R}^{n}(M,Q))\rightarrow 0,$$
which induce the exact sequences of complexes $$0\longrightarrow \textbf{Z}(Q)\otimes_{R\ltimes_{n}M}\Delta\longrightarrow \textbf{H}(Q)\otimes_{R\ltimes_{n}M}\Delta\longrightarrow[\coprod_{i=1}^{n}\Hom_{R}^{i}(M,Q),w_{1}]\otimes_{R\ltimes_{n}M}\Delta\longrightarrow 0,$$  {\Small $$0\rightarrow\textbf{Z}(\Hom_{R}(M,Q))\otimes_{R\ltimes_{n}M}\Delta\rightarrow [\coprod_{i=1}^{n}\Hom_{R}^{i}(M,Q),w_{1}]\otimes_{R\ltimes_{n}M}\Delta \rightarrow[\coprod_{i=2}^{n}\Hom_{R}^{i}(M,Q),w_{2}]\otimes_{R\ltimes_{n}M}\Delta\rightarrow 0,$$ $$\cdots\cdots$$
 $$0\rightarrow\textbf{Z}(\Hom_{R}^{n-1}(M,Q))\otimes_{R\ltimes_{n}M}\Delta\rightarrow[\coprod_{i=n-1}^{n}\Hom_{R}^{i}(M,Q),w_{n-1}]\otimes_{R\ltimes_{n}M}\Delta
 \rightarrow\textbf{Z}(\Hom_{R}^{n}(M,Q))\otimes_{R\ltimes_{n}M}\Delta\rightarrow 0.$$}
Since $fd(_{R}M^{\otimes_{R}i})<\infty$, we have $id(\Hom_{R}^{i}(M,Q)_{R})<\infty$.  So $\textbf{Z}(\Hom_{R}^{i}(M,Q))\otimes_{R\ltimes_{n}M}\Delta\cong\Hom_{R}^{i}(M,Q)\otimes_{R}\Xi$ is exact for any $0\leq i\leq n$. Thus $\textbf{H}(Q)\otimes_{R\ltimes_{n}M}\Delta$ is exact.

It follows that  $(X,f)$ is a $PGF$ left $R\ltimes_{n}M$-module.
\end{proof}
\begin{cor}\label{cor: 4.6}Assume that  $fd(_{R}M^{\otimes_{R}i})<\infty$ and $fd(M^{\otimes_{R}i}_{R})<\infty$ for any $1\leq i\leq n$, $fd(_{R\ltimes_{n}M}\textbf{Z}(R))<\infty$, $fd(\textbf{Z}(R)_{R\ltimes_{n}M})$ or $id(\textbf{Z}(R)_{R\ltimes_{n}M})<\infty$. Then
\begin{enumerate}\item $\textbf{T}(X)$ is a $PGF$ left $R\ltimes_{n}M$-module if and only if  $X$ is a $PGF$ left $R$-module.\item $\textbf{Z}(X)$ is a $PGF$ left $R\ltimes_{n}M$-module if and only if $M\otimes_R X =0$ and $X$ is a $PGF$ left $R$-module.
\end{enumerate}
\end{cor}
Applying  Theorems \ref{thm: 4.1}, \ref{thm: 4.2}, \ref{thm: 4.4} and \ref{thm: 4.5} to tensor rings, we have
\begin{cor} \label{cor: 4.7}Let $M$ be an $n$-nilpotent  $R$-bimodule and  $(X,f)$ a left $T_{R}(M)$-module.
\begin{enumerate}\item If $fd(_{T_{R}(M)}\textbf{Z}(R))<\infty$, $fd(\textbf{Z}(R)_{T_{R}(M)})$ or $id(\textbf{Z}(R)_{T_{R}(M)})<\infty$ and $(X,f)$ is a  Gorenstein flat left $T_{R}(M)$-module, then $f$ is a monomorphism  and $\coker(f)$ is a Gorenstein flat left $R$-module.\item If  $T_{R}(M)$ is a right coherent ring, $fd(_{R}M^{\otimes_{R}i})<\infty$, $pd(M^{\otimes_{R}i}_{R})<\infty$ for any $1\leq i\leq n$, $f$ is a monomorphism and $\coker(f)$ is a Gorenstein flat left $R$-module, then $(X,f)$ is a Gorenstein flat left $T_{R}(M)$-module.
\end{enumerate}
\end{cor}
\begin{cor} \label{cor: 4.8}Let $M$ be an $n$-nilpotent  $R$-bimodule and  $(X,f)$ a left $T_{R}(M)$-module.
\begin{enumerate}\item  If $fd(_{T_{R}(M)}\textbf{Z}(R))<\infty$, $fd(\textbf{Z}(R)_{T_{R}(M)})$ or $id(\textbf{Z}(R)_{T_{R}(M)})<\infty$  and $(X,f)$ is a  $PGF$ left $T_{R}(M)$-module, then $f$ is a monomorphism  and $\coker(f)$ is a $PGF$ left $R$-module.\item If   $fd(_{R}M^{\otimes_{R}i})<\infty$, $fd(M^{\otimes_{R}i}_{R})<\infty$ for any $1\leq i\leq n$, $f$ is a monomorphism and $\coker(f)$ is a $PGF$ left $R$-module, then $(X,f)$ is a $PGF$ left $T_{R}(M)$-module.
\end{enumerate}
\end{cor}
\begin{rem} \label{rem: 4.9} {\rm We mention that  Corollary \ref{cor: 4.7} generalizes \cite[Lemma 3.7 and Theorem 3.8]{DLST},
Corollary \ref{cor: 4.8} generalizes \cite[Lemmas 2.4 and 2.5]{TW}.}
\end{rem}

\bigskip
\centerline {\bf ACKNOWLEDGEMENTS}
\bigskip
This research was supported by NSFC (12271249).

\end{document}